\documentclass[psamsfonts]{amsart}  
\usepackage{amssymb}

\usepackage[T1]{fontenc}

\usepackage[dvips]{graphicx}

\usepackage{graphicx}
\usepackage{amssymb}                       
\usepackage{amsmath}
\usepackage{color}
\usepackage{times}

\usepackage[unicode,bookmarks,colorlinks]{hyperref}
\hypersetup{
   linkcolor=brickred,
}

\definecolor{mahogany}{cmyk}{0, 0.77, 0.87, 0}
\definecolor{salmon}{cmyk}{0, 0.53, 0.38, 0}
\definecolor{melon}{cmyk}{0, 0.46, 0.50, 0}
\definecolor{yellowgreen}{cmyk}{0.44, 0, 0.74, 0}
\definecolor{brickred}{cmyk}{0, 0.89, 0.94, 0.28}
\definecolor{OliveGreen}{cmyk}{0.64, 0, 0.95, 0.40}
\definecolor{RawSienna}{cmyk}{0, 0.72, 1.0, 0.45}
\definecolor{ZurichRed}{rgb}{1, 0, 0} 

\usepackage{fancyhdr}
\usepackage{amsmath,amstext,amssymb,amsopn,amsthm}
\usepackage{amsmath,amssymb,amsthm}
\usepackage[mathscr]{eucal}

\newcommand{\p}{{\mathbb{P}}}

\reversemarginpar

\definecolor{rb}{rgb}{0.1,0.2, 0.7}

\begin{document}

\newtheorem{lemma}[thm]{Lemma}
\newtheorem{proposition}{Proposition}
\newtheorem{theorem}{Theorem}[section]
\newtheorem{deff}[thm]{Definition}
\newtheorem{case}[thm]{Case}
\newtheorem{prop}[thm]{Proposition}
\newtheorem{example}{Example}

\newtheorem{corollary}{Corollary}

\theoremstyle{definition}
\newtheorem{remark}{Remark}

\numberwithin{equation}{section}
\numberwithin{definition}{section}
\numberwithin{corollary}{section}

\numberwithin{theorem}{section}

\numberwithin{remark}{section}
\numberwithin{example}{section}
\numberwithin{proposition}{section}

\newcommand{\gap}{\lambda_{2,D}^V-\lambda_{1,D}^V}
\newcommand{\gapR}{\lambda_{2,R}-\lambda_{1,R}}
\newcommand{\bD}{\mathrm{I\! D\!}}
\newcommand{\calD}{\mathcal{D}}
\newcommand{\calA}{\mathcal{A}}

\newcommand{\conjugate}[1]{\overline{#1}}
\newcommand{\abs}[1]{\left| #1 \right|}
\newcommand{\cl}[1]{\overline{#1}}
\newcommand{\expr}[1]{\left( #1 \right)}
\newcommand{\set}[1]{\left\{ #1 \right\}}

\newcommand{\calC}{\mathcal{C}}
\newcommand{\calE}{\mathcal{E}}
\newcommand{\calF}{\mathcal{F}}
\newcommand{\Rd}{\mathbb{R}^d}
\newcommand{\BR}{\mathcal{B}(\Rd)}
\newcommand{\R}{\mathbb{R}}
\newcommand{\T}{\mathbb{T}}
\newcommand{\D}{\mathbb{D}}

\newcommand{\al}{\alpha}
\newcommand{\RR}[1]{\mathbb{#1}}
\newcommand{\bR}{\mathrm{I\! R\!}}
\newcommand{\ga}{\gamma}
\newcommand{\om}{\omega}
\newcommand{\A}{\mathbb{A}}
\newcommand{\bH}{\mathbb{H}}

\newcommand{\bb}[1]{\mathbb{#1}}
\newcommand{\bI}{\bb{I}}
\newcommand{\bN}{\bb{N}}

\newcommand{\uS}{\mathbb{S}}
\newcommand{\M}{{\mathcal{M}}}
\newcommand{\calB}{{\mathcal{B}}}

\newcommand{\W}{{\mathcal{W}}}

\newcommand{\m}{{\mathcal{m}}}

\newcommand {\mac}[1] { \mathbb{#1} }

\newcommand{\bC}{\Bbb C}

\newtheorem{rem}[theorem]{Remark}
\newtheorem{dfn}[theorem]{Definition}
\theoremstyle{definition}
\newtheorem{ex}[theorem]{Example}
\numberwithin{equation}{section}

\newcommand{\Pro}{\mathbb{P}}
\newcommand\F{\mathcal{F}}
\newcommand\E{\mathbb{E}}
\newcommand\e{\varepsilon}
\def\H{\mathcal{H}}
\def\t{\tau}

\title[Weighted estimates]{Weighted square function estimates}

\author{Rodrigo Ba\~nuelos}\thanks{R. Ba\~nuelos is supported in part  by NSF Grant
\# 0603701-DMS}
\address{Department of Mathematics, Purdue University, West Lafayette, IN 47907, USA}
\email{banuelos@math.purdue.edu}
\author{Adam Os\k ekowski}\thanks{A. Os\k ekowski is supported in part by Narodowe Centrum Nauki (Poland) grant DEC-2014/14/E/ST1/00532.}
\address{Department of Mathematics, Informatics and Mechanics, University of Warsaw, Banacha 2, 02-097 Warsaw, Poland}
\email{ados@mimuw.edu.pl}


\begin{abstract}
The paper contains the proof of $L^p$-weighted norm inequalities  for both, martingales square functions and the classical square functions in harmonic analysis of Littlewood-Paley and Lusin. Furthermore,  the bounds are completely explicit and are  optimal not only on the dependence of the characteristics of the weight but also on the dependance on $p$, as  $p\to\infty$. The proof rests on Bellman function method: the estimates are deduced from the existence of an appropriate and rather complicated function of four variables.
\end{abstract}

\maketitle
\section{Introduction}

In \cite{BO}, the authors used the Bellman function approach to give new proofs of weighted $L^2$-norm inequalities for martingales and Littlewood-Paley square functions with the optimal dependence on the $A_2$ characteristics $[w]_{A_2}$ of the weight $w$ and further explicit constants.  This paper is a continuation of the work in \cite{BO} and contains a significant improvements extending those results to the $L^p$-norm inequalities  and  addressing questions raised in that paper (see second paragraph of \S5) concerning the optimal dependence not only  on the $A_p$ characteristics $[w]_{A_p}$, but also on the constant $C_p$ which appears in the $L^p$--norm inequalities.    
As in \cite{BO}, the proofs for the Littlewood-Paley square function will depend heavily on martingale estimates. Beyond these applications, the martingale square functions inequalities are of independent interest with wide range of further implications. Some convenient references here are the works \cite{Ban1, Mey2, Mey1, Var} which exhibits the connections between martingale square functions and various classical square functions arising in harmonic analysis; see also the monographs \cite{BanMoo, Bass, Ste} for further illustration of the applications in this direction.  

Let us introduce the necessary probabilistic background and formulate our main results in the martingale context; the corresponding results for Littlewood-Paley square functions will be presented later. Assume that $(\Omega,\F,\mathbb{P})$ is a complete probability space, filtered by $(\F_t)_{t\geq 0}$, a nondecreasing right-continuous sequence of sub-$\sigma$-algebras of $\F$. Suppose in addition that $\F_0$ contains all the events of probability $0$ and that all adapted martingales have continuous paths (this is the case, for instance, for the Brownian filtration). Let $X=(X_t)_{t\geq 0}$ be an adapted, uniformly integrable continuous-path martingale and let $\langle X \rangle=(\langle X_t\rangle)_{t\geq 0}$ denote its quadratic covariance process (square function). See e.g. Dellacherie and Meyer \cite{DM} for more information on the subject. 
Let $MX=\sup_{t\geq 0}|X_t|$ be the maximal function of $X$. To introduce the appropriate analogue of martingale $A_p$ weights, assume that $W$ is an integrable random variable. This variable gives rise to the uniformly integrable martingale $(W_t)_{t\geq 0}$ given by $W_t=\E(W|\F_t)$; sometimes, for consistence, we will write $W_\infty$ instead of $W$. 
 Following Izumisawa and Kazamaki \cite{IK}, we say that $W$ satisfies Muckenhoupt's condition $A_p(\text{mart})$ (where $1<p<\infty$ is a fixed parameter), if
 \begin{equation}\label{probAp}
[W]_{A_{p}}:=\sup_{t\geq 0}\left\|W_{t}\big(\E\big[W^{-1/(p-1)}\big|\F_{t}\big]\big)^{p-1}\right\|_{\infty}<\infty.
\end{equation}
We will also need a version of this condition for $p=1$. We say that $W$ is an $A_1$ weight, if there is a finite constant $C$ such that $CW_t\geq  MW$ almost surely for all $t\geq 0$. The smallest $C$ with this property will be denoted by $[W]_{A_1}$.

Any random variable $W$ as above is a density of the measure $\mbox{d}\mathbb{Q}:=W\mbox{d}\mathbb{P}$, and can be regarded as a weight. We will use the notation
$$ \|f\|_{L^p(W)}=\left(\E |f|^pW\right)^{1/p},\qquad 1\leq p<\infty,$$
for the norm in the associated weighted $L^p$ space.  With these definitions, we will prove the following theorem  extending the results in \cite{BO}.

\begin{theorem}\label{mainthm} 
Suppose that $W$ is an $A_p$ weight and $X$ is a martingale bounded in $L^p(W)$. Then for any $1<p<\infty$ we have the estimate
\begin{equation}\label{mainin}
\|{\langle X \rangle}^{1/2}_\infty\|_{L^p(W)}\leq K_p[W]_{A_p}^{\max\{1/2,1/(p-1)\}}\|X_\infty\|_{L^p(W)},
\end{equation}
where
$$ K_p=\begin{cases}
553(p-1)^{-1} &\mbox{if }1<p\leq 3,\\
189p^{1/2} & \mbox{if }p\geq 3.
\end{cases}$$
The exponent $\max\{1/2,1/(p-1)\}$ is the best possible. Furthermore, the orders of $K_p$ as $p\to 1+$ and $p\to \infty$ are also optimal (as they are already the best possible in the unweighted case).
\end{theorem}

 As shown in \cite{BO},  Theorem \ref{mainthm} implies corresponding results for the classical Littlewood-Paley and Lusin square functions on the circle $\mathbb{T}$, $\R^n$, $n\geq 1$,  and more general Markovian semigroups.  These estimates inherit the same bounds as those in inequality \eqref{mainin}. To avoid repeating the statements and proofs in \cite{BO}, we simply state our result here for the circle and leave the others to the reader. The Littlewood-Paley $g_*$--function on the circle $\mathbb{T}=\partial \mathbb{D}$ (cf. \cite{Ban1,Zyg}) is given by
\begin{equation}\label{defgstar}
 g_*(f)(e^{i\theta})=\left(\frac{1}{\pi}\int_{\mathbb{D}}\frac{1-|z|^2}{|z-e^{i\theta}|^2}\ln\frac{1}{|z|}|\nabla u_f(z)|^2\mbox{d}z\right)^{1/2},
 \end{equation}
where $dz$ is the area measure on the plane, $f$ is an integrable function on $\mathbb{T}$ and $u_f$ stands for the harmonic extension of $f$ to the disc $\mathbb{D}$. The function $g_*(f)$ carries a lot of information about the behavior of $f$, for instance, it is well-known that for $2\leq p<\infty$ there is a finite constant $C_p$ depending only on $p$ such that
\begin{equation}\label{g*} \|g_*(f)\|_{L_p(\mathbb{T})}\leq C_p\|f\|_{L_p(\mathbb{T})}.
\end{equation} 
On contrary, for $1<p<2$ this estimate fails to hold.  In fact, using the sharp version of the unweighted inequality \eqref{mainin}, due to Davis \cite{Da}, it is shown in \cite{Ban1} that if $a_p$ is the smallest zero of the confluent hypergeometric function of parameter $p$, then 

$$C_p\leq \frac{1}{\sqrt{2}\,a_p}\leq\left(\frac{8}{\pi^2}\right)^{1/2}\left[ 4\Gamma\left(\frac{p+1}{2}\right)\right]^{1/p},$$
which is of order $O(\sqrt{p})$ as $p\to\infty$, and that this order is best possible.  This result also holds on $\R^n$ with the same constant.

 In recent years, questions about the weighted version of such estimate gained  a lot of interest in the literature;  we refer to \cite{BO} and references therein for this large literature.  In what follows, the  word ``weight'' refers to a locally integrable, positive function  on the underlying space (which in the above setting is the unit circle $\mathbb{T}$ with Haar measure),  usually  denoted by $w$. Given $p\in (1,\infty)$, we say that $w\in A_p(Poisson,\mathbb{T})$ ($w$ is a Poissonian $A_p$ weight), if the $A_p$ characteristics $[w]_{A_p}$, given by
\begin{equation}\label{Poisson-T}
 [w]_{A_{p,\mathbb{T}}}:=\|u_w(z)\big(u_{w^{-1/(p-1)}}(z)\big)^{p-1}\|_{L^\infty(\mathbb{D})},
\end{equation}
is finite. As discussed  in \cite{BO}, these are the probabilistic Muckenhoupt $A_p$-weights of martingales on filtration of Brownian motion killed upon leaving the unit disc.  They are, in fact, a special case of more general Markovian semigroup $A_p$-weights; see \cite{BO} for more on this. 

\begin{theorem}\label{weights}
For any $2\leq p< \infty$ the following inequality holds: 
\begin{equation}\label{g*one} \|g_*(f)\|_{L^p(w)}\leq K_p[w]_{A_{p,\mathbb{T}}}^{\max\{1/2,1/(p-1)\}}\|f\|_{L^p(w)},
\end{equation}
where
$$ K_p=\begin{cases}
553(p-1)^{-1} &\mbox{if }2\leq p\leq 3,\\
189p^{1/2} & \mbox{if }p\geq 3.
\end{cases}$$
The exponent $\max\{1/2,1/(p-1)\}$ is the best possible. Furthermore, the order of $K_p$ as  $p\to \infty$ is also optimal as it is already the best possible in the unweighted case.
\end{theorem}

The proof of Theorem \ref{mainthm} will be split into two parts. The main difficulty lies in proving the estimate in the case $p=3$ (this is when the terms $1/2$ and $1/(p-1)$ in the definition of the optimal exponent in \eqref{mainin} coincide). Our argument will use the  Bellman function method and rests on the construction of a special function enjoying certain majorization and concavity-type properties. This approach originates from the theory of optimal control (cf. \cite{Be}) and turned out to be very efficient in the study of various semimartingale inequalities, as Burkholder noticed in his pioneering works in the eighties. Then, in the mid-nineties, Nazarov Treil and Volberg (see \cite{NT,NTV}) proved that apart from the probability theory, the method can be successfully applied in several interesting contexts arising in harmonic analysis. Since then, many mathematicians used the technique in the study of numerous important estimates; the literature on the subject is extremely vast, we mention here only the bounds for BMO class studied in \cite{I,SV0,SV1}, inequalities for Muckenhoupt weights \cite{Vas}, estimates for singular integral operators \cite{N,Pet,PV} and estimates for fractional operators \cite{BO2}.

As we shall see, the function we construct is quite complicated. This is done in Section \ref{p=3}.  We strongly believe that this function  can be modified to study other interesting estimates of harmonic analysis. The second part of the proof of Theorem \ref{mainthm} uses extrapolation-type arguments that  allow us to deduce the estimate \eqref{mainin} in the full range $1<p<\infty$ from the single case $p=3$. This will be discussed in Section \ref{extrapolation} below. Let us stress here that there seems to be no analytic extrapolation argument which would be applicable directly to the function $g_*$. Indeed, otherwise we would in particular obtain the unweighted $L^p$ estimates for $g_*$ in the range $1<p<2$, which is well known to be false. 
The final part of the paper is devoted to the analytic applications of Theorem \ref{mainthm}.  In particular, we will give the proof of Theorem \ref{weights} there.  

\section{The proof of Theorem \ref{mainthm} for $p=3$}\label{p=3}

\subsection{On the method of proof} For the sake of clarity, we have decided to start with a description of the approach which will lead us to the desired estimate \eqref{mainin}. The detailed, quite elaborate calculations are postponed to the second part of this section.

Let us start with the following useful interpretation of $A_p$ weights, valid for $1<p<\infty$. Fix such a weight $W$ and suppose that $c\geq [W]_{A_p}$. Let $V=(V_t)_{t\geq 0}$ be the martingale given by $V_t=\E(W^{1/(1-p)}|\F_t)$, $t\geq 0$. Note that Jensen's inequality implies $W_t V_t^{p-1}\geq 1$ almost surely for any $t\geq 0$. Furthermore,  the $A_p$ condition is equivalent to the existence of a positive constant $c$ such that 
$$ W_t V_t^{p-1}\leq c,\qquad \mbox{with probability }1.$$
In other words, an $A_p$ weight of characteristic equal to $c$ gives rise to a two-dimensional martingale $(W,V)$ taking values in the domain
$$ \{(w,v)\in (0,\infty)\times (0,\infty)\,:\,1\leq wv^{p-1}\leq c\}.$$
In addition, this martingale terminates at the lower boundary of this domain: $W_\infty V_\infty^{p-1}=1$ almost surely. A nice fact is that this provides a full characterization of $A_p$ weights: given any martingale pair $(W,V)$ (with continuous-path $W$ of mean $1$) taking values in the above domain and terminating at the set $wv^{p-1}=1$, one easily checks that its first coordinate is an $A_p$ weight with $[W]_{A_p}\leq c$.

We are ready to describe the idea behind the proof of \eqref{mainin} for $p=3$. We want to establish the inequality
\begin{equation}\label{wewant}
 \E\bigg[ \bigg(\langle X\rangle_t^{3/2}- 64^3[W]_{A_3}^{3/2} |X_t|^3\bigg)W\bigg]\leq 0,\qquad t\geq 0,
\end{equation}
for any weight $W$ satisfying Muckenhoupt's condition $(A_3)$ and any continuous-path martingale $X=(X_t)_{t\geq 0}$. To this end, we will construct the Bellman function associated with this problem. For a more extensive description of the Bellman function method used here, see \cite{Os}.  Fixed $c\geq 1$ and consider the domain
$$ \mathcal{D}_c=\{(x,y,w,v):x\in \R,\,y\geq 0,\,1\leq wv^2\leq c\}.$$
Suppose that there is a $C^2$ function $B:\mathcal{D}_c\to \R$ satisfying the following properties:

\smallskip

1$^\circ$ (Initial condition) We have $B(x,x^2,w,v)\leq 0$ if $x\in \R$ and $1\leq wv^2\leq c$.

2$^\circ$ (Majorization property) There is a positive constant $\kappa>0$ such that
\begin{equation}\label{maj}
B(x,y,w,v)\geq \kappa(y^{3/2}w-64^3c^{3/2}|x|^3w)\qquad \mbox{for }(x,y,w,v)\in \mathcal{D}_c.
\end{equation}

3$^\circ$ (Concavity-type property) For any $(x,y,w,v)\in \mathcal{D}_c$ with $1< wv^2\leq c$ and any $d,\,r,\,s\in \R$, the function
$$ \xi_B(t):=B(x+td,y+t^2d^2,w+tr,v+ts),$$
given for those $t$, for which $1\leq (w+tr)(w+ts)^2\leq c$, satisfies $\xi_B''(0)\leq 0$.

\smallskip

The connection between the existence of such a function and the validity of \eqref{wewant} is described in the following lemma. 

\begin{lemma}
Let $c\geq 1$ be fixed. If $B$ satisfies the conditions 1$^\circ$, 2$^\circ$ and 3$^\circ$, then the inequality \eqref{wewant} holds true for all weights $W$ with $[W]_{A_3}\leq c$.
\end{lemma}
\begin{proof}
The claim follows by a straightforward use of It\^o's formula. Since $B$ is of class $C^2$, we may write
$$ B(X_t,\langle X\rangle_t,W_t,V_t)=I_0+I_1+I_2+I_3/2,$$
where
\begin{align*}
I_0&=B(X_0,\langle X\rangle_0,W_0,V_0),\\
 I_1&=\int_{0+}^t B_x(X_s,\langle X\rangle_s,W_s,V_s)\mbox{d}X_s+\int_{0+}^t B_w(X_s,\langle X\rangle_s,W_s,V_s)\mbox{d}W_s\\
&\quad +\int_{0+}^t B_v(X_s,\langle X\rangle_s,W_s,V_s)\mbox{d}V_s,\\
I_2&=\int_{0+}^t B_y(X_s,\langle X\rangle_s,W_s,V_s)\mbox{d}\langle X\rangle_s,\\
I_3&=\int_{0+}^t D^2_{x,w,v}B(X_s,\langle X\rangle_s,W_s,V_s)\mbox{d}\langle X,W,V\rangle_s.
\end{align*}
Here in the definition of $I_3$ we have used a short notation for the sum of second-order terms: more formally, $I_3$ equals
$$\int_{0+}^t B_{xx}(X_s,\langle X\rangle_s,W_s,V_s)d\langle X\rangle_s+2\int_{0+}^t B_{xw}(X_s,\langle X\rangle_s,W_s,V_s)\mbox{d}\langle X,W\rangle_s+\ldots$$
and so on. Let us study the properties of the terms $I_0$ through $I_3$. We have $I_0\leq 0$ because of the condition 1$^\circ$. Furthermore, the expectation of $I_1$ is zero, by the properties of stochastic integrals. Finally, the condition 3$^\circ$ implies that $I_2+I_3\leq 0$. Consequently, we get
$$ \E B(X_t,\langle X\rangle_t,W_t,V_t)\leq 0.$$
It remains to apply the majorization \eqref{maj} and use the equality $W_t=\E(W|\F_t)$ to get the assertion.
\end{proof}

Actually, we will construct a function which will satisfy 1$^\circ$, 2$^\circ$ and a certain weaker form of 3$^\circ$ (in particular, $B$ will not be of class $C^2$). Let us discuss this issue briefly and explain why this will not alter the assertion of the above lemma.  Assume that $D_1$, $D_2$, $D_3$ are the ``angular'' subsets of $\mathcal{D}_c$ given by
\begin{equation}\label{defDi}
\begin{split}
D_1&=\left\{(x,y,w,v): y^{1/2}\geq 16c^{1/2}|x|(c/t)^{1-\beta}\right\},\\
D_2&=\left\{(x,y,w,v): 4|x|\leq y^{1/2}< 16c^{1/2}|x|(c/t)^{1-\beta}\right\},\\
D_3&=\left\{(x,y,w,v): y^{1/2}\leq 4|x|\right\},
\end{split}
\end{equation}
where $t=wv^2$. 
We will construct three functions $B_1,\,B_2,\,B_3:\mathcal{D}_c\to \R$ of class $C^2$ which satisfy the following set of requirements:

\begin{itemize}
\item[(a)] $B_i$ satisfies 3$^\circ$ for $(x,y,w,v)\in D_i$, $i=1,\,2,\,3$.
\item[(b)] $B_1\leq B_2$ on $D_1$, $B_2\leq \min(B_1,B_3)$ on $D_2$ and $B_3\leq B_2$ on $D_3$. 
\end{itemize}
Then the function
\begin{equation}\label{defB}
 B=B_i\quad \mbox{on }D_i,\quad i=1,\,2,\,3,
\end{equation}
is continuous and satisfies
$$ \E B(X_t,\langle X\rangle_t,W_t,V_t)\leq \E B(X_0,\langle X\rangle_0,W_0,V_0).$$
This can be proved, for example, by applying the more general  It\^o formula with local times on surfaces (see Peskir \cite{Pe}) or, alternatively, using standard mollification arguments. By 1$^\circ$ and 2$^\circ$, this will yield the validity of \eqref{wewant}.

Let us quickly comment on the conditions 1$^\circ$, 2$^\circ$ and 3$^\circ$. The first two properties are very easy, the main difficulty lies in the proof of 3$^\circ$. Roughly speaking, this condition means that for each $(x,y,w,v)$, a certain quadratic form of variables $d$, $r$ and $s$ is nonpositive. This then will lead us to the verification of the nonpositive-definiteness of the associated Hessian-type matrices of $B_1$, $B_2$ and $B_3$. As we shall see in \eqref{defBi} below, the functions $B_i$ will be built from several simpler ``blocks'', for which the study of the corresponding quadratic forms is simpler.

The formula \eqref{defBi} for the special functions looks very mysterious and hence, before we proceed to the formal verification, let us give the reader some intuition about the search for this object and the problems that arise. The reasoning below will be very informal and vague; the reader may skip it and proceed directly to the next subsection. As we have already mentioned above, the key lies in an appropriate handling of the concavity property 3$^\circ$. A natural starting point is the unweighted setting. In this case,  $W$ and $V$ are constant and the function $B$ depends on only two variables $x$ and $y$. The sharp $L^3$ estimate $||[X,X]_t^{1/2}||_3\leq \sqrt{3}||X_t||_3$ (cf. Davis \cite{Da}) can be obtained with the use of the function
$$ B(x,y)=y^{3/2}-3x^2y^{1/2},\qquad x\in \R,\,y\geq 0,$$
which is built from two blocks: the ``convex''  part $y^{3/2}$ and ``concave'' part $-x^2y^{1/2}$ multiplied by $3$. This function suggests considering
$$ B(x,y,w,v)=y^{3/2}w -Kx^2y^{1/2}w$$
in the weighted case, for some positive constant $K$ possibly depending on $c$. A direct verification of 3$^\circ$ brings us to the derivative
$$ \xi_B''(0)=(3-2K-Kx^2y^{-1})y^{1/2}wd^2-4Kxy^{1/2}dr,$$
which, unfortunately, can be positive.  This is due to the appearance of the summand $dr$ and the lack of a term of the form $ar^2$, for some $a<0$. To handle this difficulty, we will correct slightly the term $y^{3/2}w$ in the definition of $B$. Replacing it with $y^{3/2}F(w,v)$ for an appropriately chosen function $F$ makes $\xi_B''(0)$ behave like (the expressions are slightly different, the formal calculations are postponed to the next subsection)
$$ (3-2K-Kx^2y^{-1})y^{1/2}wd^2-4Kxy^{1/2}dr-c^{-1}y^{3/2}w^{-1}r^2,$$
i.e., the additional term $-c^{-1}y^{3/2}w^{-1}r^2$ emerges. So, if we take $K=2$ and restrict ourselves to $y^{1/2}\geq C|x|$ (for appropriately chosen $C$), then $B$ will have the required concavity. To ensure 3$^\circ$ for $y^{1/2}\leq C|x|$, we need to add a different term which will ``overpower'' the summand $dr$. To this end, we will use a term of the form $L_1x^2y^{1/2}w^{1-\beta}v^{-2\beta}-L_2|x|^3/v^2$, for certain $L_1,\,L_2>0$ and $\beta\in (0,1)$. Note that the function $w\mapsto w^{1-\beta}$ is concave, which will guarantee the appearance of the term $ar^2$ (for some negative $a$) in $\xi_B''(0)$. The ``correction'' $-L_2|x|^3/v^2$ is necessary to control other terms which arise. However, the function we thus obtain does not satisfy 3$^\circ$ if $y^{1/2}\leq 4|x|$ (i.e., $y^{1/2}$ is of comparable size to $x$ or smaller). To handle this, we will introduce  yet another formula for such $x$ and $y$, of the form 
$$ B(x,y,w,v)=y^{3/2}C_1(w,v)-|x|^{3/2}C_2(w,v).$$
If we take all the coefficients appropriately, then the resulting function $B$ is continuous and satisfies the structural properties (a) and (b) above. Although the final formula for $B$ will involve some additional terms, the above discussion illustrates quite well the steps of construction that lead us to the desired object.

\subsection{A special function} Throughout this subsection, $c\geq 1$ is a fixed parameter and, for notational convenience and brevity, we set $a=1/2$, $\alpha=1-(4c)^{-1}$, $\beta=3/4$. Recall the sets $D_1$, $D_2$ and $D_3$ given in \eqref{defDi}. In what follows, if $w,\,v$ are given positive numbers, then $t$ will stand for the product $wv^2$.

Define the functions $b_i$ by
\begin{align*}
 b_1(x,y,w,v)&=y^{3/2}(wv^2-a)^\alpha v^{-2},\\
b_2(x,y,w,v)&=c^\beta x^2y^{1/2}w^{1-\beta}v^{-2\beta}-80c^{3/2}|x|^3v^{-2},\\
b_3(x,y,w,v)&=64c^{3/2}|x|^3v^{-2},\\
b_4(x,y,w,v)&=c^{3/2}|x|^3v^{-2}\big(6-4(wv^{2})^{-1/2}-c^{-1/2}\ln(wv^2)\big)^{-2},\\
 b_5(x,y,w,v)&=c^\beta y^{3/2}w^{1-\beta}v^{-2\beta}.
 \end{align*}
The Bellman function $B$ is given by \eqref{defB}, where (for brevity, we skip the evaluation of the functions at the point $(x,y,w,v)$)
\begin{equation}\label{defBi}
 \begin{split}
B_1&=b_1- 2x^2y^{1/2}w-192b_3-276480b_4,\\
B_2&=b_1- 2x^2y^{1/2}w+192b_2-276480b_4,\\
B_3&=b_1-y^{3/2}w/8-240b_3-276480b_4+12b_5.
\end{split}
\end{equation}

We start with some preliminary properties of $B$.

\begin{lemma}
The function $B$ satisfies the properties 1$^\circ$, 2$^\circ$ and (b). 
\end{lemma}
\begin{proof}
We start with the property (b). If $(x,y,w,v)\in D_1$, then
$$b_2(x,y,w,v)\geq 16c^{3/2}|x|^3v^{-2}-80c^{3/2}|x|^3v^{-2}=-b_3(w,y,w,v),$$
and hence $B_2\geq B_1$. If $(x,y,w,v)\in D_2$, then the above inequality reverses and hence $B_2\leq B_1$. Furthermore, on $D_2$ we have
\begin{align*}
 -2x^2y^{1/2}w+192b_2(x,y,w,v)&=\left(-2+192\left(\frac{c}{t}\right)^\beta\right)x^2y^{1/2}w-240b_3(x,y,w,v)\\
 &\leq \left(-2+192\left(\frac{c}{t}\right)^\beta\right)\frac{y^{3/2}}{16}w-240b_3(x,y,w,v),
 \end{align*}
which is equivalent to $B_2\leq B_3$. Finally, on $D_3$, the above estimate is reversed and hence (b) holds.

Next, we turn our attention to the initial condition 1$^\circ$. Any point of the form $(x,x^2,w,v)$ belongs to $D_3$, which implies the equality $B(x,x^2,w,v)=B_3(x,x^2,w,v)$. We have $(wv^2-a)^\alpha v^{-2}\leq w$, so $b_1\leq b_5$ and 1$^\circ$ follows from
\begin{align*}
 B_3(x,x^2,w,v)&\leq 13b_5(x,x^2,w,v)-240b_3(x,x^2,w,v)\\
&\leq \left(13\left(\frac{c}{t}\right)^\beta -240\cdot 64\sqrt{c}\cdot\frac{c}{t}\right)|x|^3w\leq 0.
\end{align*}
To verify the majorization 2$^\circ$, we start with the observation that $(wv^2-a)^\alpha v^{-2}\geq w/2$. Indeed, multiplying by $v^2$ and putting all the terms on the left transforms the estimate into $(t-a)^\alpha-t/2\geq 0$; now, the left hand side, as a function of $t\in [1,c]$, is increasing and positive at $t=1$. Therefore, we see that $ b_1(x,y,w,v)\geq y^{3/2}w/2$ and hence, using Young's inequality,
$$ b_1(x,y,w,v)-2x^2y^{1/2}w\geq y^3w/2-\big(y^3w/12+32|x|^3w/3\big)\geq \frac{1}{3}(y^{3/2}w-32|x|^3w).$$
Observe that 
$$192b_2(x,y,w,v)\geq -192\cdot 80c^{3/2}|x|^3v^{-2}=-240b_3(x,y,w,v)\geq -240\cdot 64c^{3/2}|x|^3w,$$ 
$$b_4(x,y,w,v)\leq c^{3/2}|x|^3w/4,$$ since $6-4t^{-1/2}-c^{-1/2}\ln t\geq 2$ for any $t\in [1,c],$
and $$12b_5(x,y,w,v)\geq y^{3/2}w/8.$$ Combining all these estimates gives
\begin{align*}
 B(x,y,w,v)&\geq \frac{1}{3}(y^{3/2}w-32x^3w)-240\cdot 64c^{3/2}|x|^3w-276480c^{3/2}|x|^3w/4\\
 &\geq \frac{1}{3}\left(y^{3/2}w-253472c^{3/2}|x|^3w\right)\\
 &\geq \frac{1}{3}\left(y^{3/2}w-(64c^{1/2})^3|x|^3w\right),
\end{align*}
which is the desired majorization.
\end{proof}

It remains to verify that the function $B$ enjoys the concavity property (a). This will be done in the sequence of four lemmas below.

\begin{lemma}\label{lem1}
(i) Let $\varphi(t)=(t-a)^\alpha$ for $t\geq a$ and let $F(w,v)=\varphi(wv^2)v^{-2}$. Then for any $w,v>0$ satisfying $1\leq wv^2\leq c$ we have
$$ D^2F(w,v)\leq \left[\begin{array}{cc}
v^2\varphi''(wv^2)/2 & 0 \\ 0 & 0 
\end{array}\right],$$
where $D^2F$ stands for the Hessian matrix of $F$.

(ii) If $b(x,y,w,v):=b_1(x,y,w,v)-2x^2y^{1/2}w,$ then $\xi_{b}''(0)\leq 0$ on $D_1$.
\end{lemma}

\begin{rem}\label{rem1}
This lemma handles the property 3$^\circ$ of the Bellman function $B$ on the domain $D_1$. Indeed, the additional summands $-100b_3$ and $-192\cdot 128 b_5$ appearing in the definition of $B$ on $D_1$ are concave functions of $x$, $w$ and $v$ (see Lemma \ref{lem2} below) and hence do not affect this property. We need these summands only for the sake of the structural property (b).
\end{rem}
\begin{proof}[Proof of Lemma \ref{lem1}]
(i) We recall that $t=wv^2$. The claim is equivalent to
$$ \left[\begin{array}{cc}
v^2\varphi''(t)/2 & 2wv \varphi''(t)\\
2wv\varphi''(t) & 4w^2\varphi''(t)+6v^{-4}\big(\varphi(t)-t\varphi'(t)\big)
\end{array}\right]\leq 0.$$
 Clearly, $\varphi$ is a concave function, so it is enough to prove that
$$ 6v^{-4}\big(\varphi(t)-t\varphi'(t)\big)\leq 4w^2\varphi''(t),$$
since then the determinant of the above matrix will be nonnegative. The inequality can be rewritten in the form
$ 3(t-a)(a+t(\alpha-1))\geq 2\alpha(1-\alpha)t$, or equivalently,
$$ 3\left(t-\frac{1}{2}\right)\left(c-\frac{t}{2}\right)-\left(1-\frac{1}{4c}\right)t\geq 0.$$
The left-hand side is an increasing function of $c\in [t,\infty)$, so it is enough to check the estimate for $c=t$. But then it is equivalent to the trivial bound $(6t-1)(t-1)\geq 0$.

(ii) We have
\begin{align*}
\xi_{b}''(0)=&\,y^{3/2}\big\langle D^2F(w,v)(r,s),(r,s)\big\rangle\\
&+3y^{1/2}F(w,v)d^2-4y^{1/2}wd^2-8xy^{1/2}dr-2x^2y^{-1/2}wd^2,
\end{align*}
where $\langle\cdot,\cdot \rangle$ denotes the standard scalar product. 
By the inequalities $-2x^2y^{-1/2}wd^2\leq 0$, $F(w,v)\leq w$ and part (i) of the lemma, we get
\begin{equation}\label{weget}
 \xi_{b}''(0)\leq \frac{1}{2}\alpha(\alpha-1)y^{3/2}(wv^2-a)^{\alpha-2}v^2r^2-y^{1/2}wd^2-8xy^{1/2}dr.
\end{equation}
Hence, to show that $\xi_{b}''(0)\leq 0$, it suffices to prove that the discriminant of the right-hand side above, treated as a function of $d$, is nonpositive:
$$ 64x^2yr^2+4\cdot \frac{1}{2}\alpha(\alpha-1)y^2(wv^2-a)^{\alpha-2}wv^2r^2\leq 0,$$
or, equivalently, 
$$ 128 cx^2\leq \left(1-\frac{1}{4c}\right)t(t-a)^{\alpha-2}y.$$
But $1-(4c)^{-1}\geq 3/4$ and $t(t-a)^{\alpha-2}\geq (t-a)^{\alpha-1}\geq t^{\alpha-1}\geq c^{-1/(4c)}\geq 2/3$. Since $y\geq 256cx^2$ (guaranteed by the assumption $(x,y,w,v)\in D_1$), we are done.
\end{proof}

The most technical part is the analysis of $b_2$. Here is the precise statement.

\begin{lemma}\label{lem3}
If $4 |x|\leq y^{1/2}\leq 16(c/t)^{1-\beta}\sqrt{c}|x|$ and $|\gamma|\leq 1/24$, then
$$
 \xi_{b_2}''(0)\leq \gamma \left(\frac{c}{t}\right)^\beta xy^{1/2} dr.
$$
\end{lemma}

\begin{rem}\label{rem2}
This lemma gives the property 3$^\circ$ of the function $B$ on $D_2$. Indeed, it implies that $\xi_{b_2}''(0)$ (when multiplied by the factor $192$) ``overpowers'' the problematic term $-8xy^{1/2}dr$ arising from the part $b_1-2x^2y^{1/2}w$: see \eqref{weget} above. As in Remark \ref{rem1} above, the summand $-100b_3$ in the definition of $B$ on  $D_2$ does not affect the concavity. 
\end{rem}
\begin{proof}[Proof of Lemma \ref{lem3}]
We compute that
$$ \xi_{b_2}''(0)=\left(\frac{c}{t}\right)^\beta y^{1/2}\bigg\langle A_1(d,r,s),(d,r,s)\bigg\rangle-\bigg\langle A_2(d,r,s),(d,r,s)\bigg\rangle,$$
where
$$ A_1=\left[\begin{array}{ccc}
2w+x^2y^{-1}w & 2(1-\beta)x & -4\beta xwv^{-1}\\
2(1-\beta)x & \beta(\beta-1)x^2w^{-1} & 2\beta(\beta-1)x^2v^{-1}\\
-4\beta xwv^{-1} & 2\beta(\beta-1)x^2v^{-1} & 2\beta(2\beta+1)x^2wv^{-2}
\end{array}\right]$$
and
$$ A_2=80c^{3/2}|x|\left[\begin{array}{ccc}
6v^{-2} & 0 & -6xv^{-3}\\
0 & 0 & 0\\
-6xv^{-3} & 0 & 6x^2v^{-4}
\end{array}\right].$$
It is easy to see that the matrix $A_2$ is nonnegative-definite. Consequently, since $16c^{3/2}|x|\geq (c/t)^\beta y^{1/2}t$, we see that
$$ A_2\geq \left(\frac{c}{t}\right)^\beta y^{1/2} A_3:=\left(\frac{c}{t}\right)^\beta y^{1/2}\left[\begin{array}{ccc}
30 w & 0 & -30 xwv^{-1}\\
0 & 0 & 0\\
-30 xwv^{-1} & 0 & 30 x^2wv^{-2}
\end{array}\right].$$
Therefore, the claim will be proved if we show that the matrix
$$ A_1-A_3+\left[\begin{array}{ccc}
0 & -\gamma x & 0\\ -\gamma x & 0 & 0 \\ 0 & 0 & 0
\end{array}\right]$$
is nonpositive-definite. Performing standard operations on rows and columns of this matrix, we may get rid of almost all variables $x$, $y$, $w$ and $v$, obtaining
$$ \left[\begin{array}{ccc}
x^2y^{-1}-28 & -\gamma+1/2 & -27\\
-\gamma+1/2 & -3/16 & -3/8\\
-27 & -3/8 & -30+15/4
\end{array}\right].$$
To check that this matrix is nonpositive-definite, we use Sylvester's criterion. We have $-30+15/4<0$ and
$$ \operatorname*{det}\left[\begin{array}{cc}
-3/16 & -3/8\\
-3/8 & -30+15/4
\end{array}\right]>0,$$
so it remains to check that the determinant of the full matrix is nonpositive. Let us add the first column of this matrix to the last column, and then the first row to the third row. The determinant does not change after these operations and hence it is equal to
$$ \operatorname*{det}\left[\begin{array}{ccc}
x^2y^{-1}-28 & -\gamma+1/2 & -1+x^2y^{-1}\\
-\gamma+1/2 & -3/16 & -\gamma+1/8\\
-1+x^2y^{-1} & -\gamma+1/8 & x^2y^{-1}-1/4
\end{array}\right].$$
We compute this determinant using Sarrus' rule, expanding it into a sum of six products. Let us group these products appropriately. Observe that
$$ 2(-1+x^2y^{-1})(-\gamma+1/2)(-\gamma+1/8)-(-\gamma+1/2)^2(x^2y^{-1}-1/4)<0,$$
since $-\gamma+1/2>0$, $-1+x^2y^{-1}\leq 4(x^2y^{-1}-1/4)\leq 0$ and $-\gamma+1/8\geq (-\gamma+1/2)/8$. Therefore it suffices to check that
$$ (x^2y^{-1}-28)\left[-\frac{3}{16}\left(x^2y^{-1}-\frac{1}{4}\right)-\left(-\gamma+\frac{1}{8}\right)^2\right]+\frac{3}{16}(-1+x^2y^{-1})^2\leq 0.$$
However, we have $x^2y^{-1}-1/4\leq 1/16-1/4=-3/16$ and $|-\gamma+1/8|\leq 1/8+1/24=1/6$, so the expression in the square brackets is not smaller than $9/256-1/36>1/144$ and hence the full expression on the left-hand side above is not bigger than
$$ (x^2y^{-1}-28)\cdot \frac{1}{144}+\frac{3}{16}(-1+x^2y^{-1})^2\leq -\frac{27}{144}+\frac{3}{16}=0.$$
This completes the proof.
\end{proof}

The final two lemmas concern the behavior of $B$ on $D_3$. Let us first study (jointly) the functions $b_3$ and $b_4$. As they appear with some negative coefficients in the definition of $B$, we need an appropriate \emph{lower} bound for the functional $\xi$.

\begin{lemma}\label{lem2}
For any $(x,y,w,v)\in \mathcal{D}_c$ we have 
$$\xi_{b_4}''(0)\geq \frac{c|x|^3s^2}{72v^4}\quad \mbox{and}\quad \xi_{b_3+1152b_4}''(0)\geq \frac{2}{3}c|x|v^{-2}d^2+15c|x|^3v^{-4}s^2.$$
\end{lemma}
\begin{proof} It will be convenient to introduce the function
\begin{align*}
 G(w,v)&=v\big(6-4(wv^{-2})^{-1/2}-c^{-1/2}\ln(wv^2)\big)\\
&=6v-4w^{-1/2}-c^{-1/2}v\ln(wv^2).
\end{align*}
Then $b_4(x,y,w,v)=c^{3/2}|x|^3G(w,v)^{-2}$. 
Let us compute the Hessian of the function $b_4/c^{3/2}$ (considered as a function of variables $x$, $w$ and $v$). It is equal to
\begin{align*}
6\!\left[\!\begin{array}{ccc}
|x|G^{-2} & -x^2G^{-3}G_w & -x^2 G^{-3}G_v\\
-x^2G^{-3}G_w & |x|^3G^{-4}(G_w)^2 & |x|^3G^{-4}G_wG_v\\
-x^2G^{-3}G_v & |x|^3G^{-4}G_wG_v & |x|^3 G^{-4}(G_v)^2
\end{array}\!\right] \! -2|x|^3G^{-3}\!\left[\begin{array}{ccc}
0 & 0 & 0\\
0 & G_{ww} & G_{wv}\\
0 & G_{vw} & G_{vv}
\end{array}\right].
\end{align*}
The first matrix is nonnegative-definite, which can be easily checked by Sylvester's criterion. To deal with the second part, we will prove that the matrix
$$ A=\left[\begin{array}{cc}
G_{ww} & G_{wv}\\
G_{vw} & G_{vv}+\frac{3}{2}c^{-1/2}v^{-1}
\end{array}\right]$$
is nonpositive-definite. We have that $G_{ww}(w,v)=c^{-1/2}w^{-5/2}(-3c^{1/2}+w^{1/2}v)$, $G_{wv}(w,v)=-c^{-1/2}w^{-1}$ and $G_{vv}(w,v)=-2c^{-1/2}v^{-1}$. 
Since $wv^2\leq c$, we get $G_{ww}<0$ and hence it suffices to check that $\operatorname*{det}A\geq 0$. This is equivalent to
$$ \frac{3c^{1/2}-w^{1/2}v}{2cw^{5/2}v}\geq \frac{1}{cw^2},$$
or $3c^{1/2}\geq 3w^{1/2}v$, which is evidently true. Combining all the observations above and multiplying by $c^{3/2}$ (the above considerations involved the function $b_4/c^{3/2}$) we obtain
$$ \xi_{b_4}''(0)=\big\langle D^2_{x,w,v}b_2(d,r,s),(d,r,s)\big\rangle\geq 3c|x|^3G^{-3}v^{-1}s^2,$$
and it remains to note that $G\leq 6v$ to complete the proof of the first inequality. The second estimate follows quickly: we compute directly that
$$ \xi_{b_3}''(0)=64c^{3/2}\left({6|x|v^{-2}d^2}-12x|x|v^{-3}ds+{6|x|^3s^2}{v^{-4}}\right),$$
so, using the previous bound for $\xi_{b_4}''(0)$ and the fact that $1152/72=16$, we get
\begin{align*}
 \xi_{b_3+1152b_4}''(0)&\geq 64c^{3/2}\left({6|x|v^{-2}d^2}-12x|x|v^{-3}ds+{6|x|^3s^2}{v^{-4}}\right)+{16c|x|^3v^{-4}s^2}\\
 &\geq \frac{2}{3}c|x|v^{-2}d^2+15c|x|^3v^{-4}s^2.\qedhere
\end{align*}
\end{proof}

We are ready for the analysis of $B$ on $D_3$.

\begin{lemma}
We have $\xi_{B_3}''(0)\leq 0$ on $D_3$.
\end{lemma}
\begin{proof}
It follows from the calculations already carried out in the proof of Lemma \ref{lem1} that
$$ \xi_{b_1}''(0)\leq 3y^{1/2}(wv^2-a)^\alpha v^{-2} d^2\leq 3y^{1/2}wd^2\leq 3(c/t)^\beta y^{1/2}wd^2.$$
Furthermore,
\begin{align*}
 \xi_{b_5}''(0)&=3(c/t)^\beta y^{1/2}w d^2+(c/t)^\beta y^{3/2}\left(-\frac{3}{16}w^{-1}r^2-\frac{3}{4}v^{-1}rs+\frac{15}{4}wv^{-2}s^2\right)\\
 &\leq 3(c/t)^\beta y^{1/2}w d^2+(c/t)^\beta y^{3/2}\cdot \frac{9}{2}wv^{-2}s^2.
 \end{align*}
Because $y^{1/2}\leq 4 |x|$, we get
\begin{align*}
 \xi''_{b_1+12b_5}(0)&\leq 39(c/t)^\beta (4|x|)wd^2+54 (c/t)^\beta (4|x|)^3 wv^{-2}s^2\\
&\leq 156c|x|v^{-2}d^2+3456c|x|^3v^{-4}s^2\\
&\leq 240\left(\frac{2}{3}c|x|v^{-2}d^2+15c|x|^3v^{-4}s^2\right),
\end{align*}
which by the previous lemma is not bigger than 
$$240\xi_{b_3+1152b_4}''(0)=\xi_{240b_3+276480b_4}''(0).$$
\end{proof}

\section{Extrapolation: proof of Theorem \ref{mainthm}}\label{extrapolation}

We follow the presentation in Duoandikoetxea \cite{Du}. The proof exploits the following structural properties of $A_p$ weights.

\begin{lemma}\label{lem31}
(i) Let $1\leq p<p_0<\infty$. If $W\in A_p(mart)$ and $U\in A_1(mart)$, then $WU^{p-p_0}\in A_{p_0}(mart)$ and
$$ [WU^{p-p_0}]_{A_{p_0}}\leq [W]_{A_p}[U]_{A_1}^{p_0-p}.$$

(ii) Let $1<p_0<p<\infty$. If  $W\in A_p(mart)$ and $U\in A_1(mart)$, then $$(W^{p_0-1}U^{p-p_0})^{1/(p-1)}\in A_{p_0}(mart)$$  and
$$ [(W^{p_0-1}U^{p-p_0})^{1/(p-1)}]_{A_{p_0}}\leq [W]_{A_p}^{(p_0-1)/(p-1)}[U]_{A_1}^{(p-p_0)/(p-1)}.$$
\end{lemma}
\begin{proof}
We will establish only the first part, the proof of the second half is analogous. For any $t\geq 0$ we have $U_\infty\geq [U]_{A_1}^{-1}MU\geq [U]_{A_1}^{-1}U_t$, so
$$ \E (W_\infty U_\infty^{p-p_0}|\F_t)\leq [U]_{A_1}^{p_0-p}\E(W_\infty|\F_t)U_t^{p-p_0}=[U]_{A_1}^{p_0-p}W_tU_t^{p-p_0}.$$
Furthermore, by H\"older's inequality,
\begin{align*}
 \E((W_\infty U_\infty^{p-p_0})^{1/(1-p_0)}|\F_t)&\leq (\E(W_\infty^{1/(1-p)}|\F_t))^{(p-1)/(p_0-1)}\E(U_\infty|\F_t)^{(p_0-p)/(p_0-1)}\\
 &= (\E(W_\infty^{1/(1-p)}|\F_t))^{(p-1)/(p_0-1)}U_t^{(p_0-p)/(p_0-1)}.
\end{align*}
Combining these two estimates gives the desired inequality.
\end{proof}

The next step is the probabilstic analogue of the so-called Rubio de Francia algorithm.  Note that in  light of Doob's inequality, for any $p>1$ the maximal operator $M$ (leading to martingale maximal function) can be treated as a sublinear operator on nonnegative random variables belonging to $L^p$. Indeed, any such random variable $f$ gives rise to the $L^p$-bounded martingale $(\E(f|\F_t))_{t\geq 0}$, whose maximal function $Mf$ is again a nonnegative random variable belonging to $L^p$. This observation is crucial for our further considerations, as Rubio de Francia algorithm involves iterations of the operator $M$ and thus it can be applied in the martingale setting and the proofs are the same as in \cite{Du}.   Here are the precise statements as in \cite[pp 1888-1892]{Du}.  We give the proofs  for the convenience of the reader. The martingale extrapolation  has been used  in \cite{DPW} in the proof of the dimensionless linear weighted bounds for the vector of Riesz transforms.

\begin{lemma}\label{lem32}
Let $p>1$. Let $W\in A_p(mart)$ and let $f$ be a nonnegative random variable belonging to $L^p(W)$. Denote by $M^k$ the $k$-th iterate of $M$, $M^0f=f$, and let $||M||_{L^p(W)}$ be the norm of $M$ as an operator on $L^p(W)$. Then the random variable
$$ Rf=\sum_{k=0}^\infty \frac{M^kf}{(2\|M\|_{L^p(W)})^k}$$
satisfies (i) $f\leq Rf$ almost surely, (ii) $\|Rf\|_{L^p(W)}\leq 2\|f\|_{L^p(W)}$ and (iii) $Rf$ is an $A_1$ weight with $[Rf]_{A_1}\leq 2\|M\|_{L^p(W)}.$
\end{lemma}
\begin{proof}
The inequality $f\leq Rf$ is evident since the term corresponding to $k=0$ is equal to $f$. The second property, (ii), follows from the estimate $||M^k||_{L^p(W)}\leq ||M||_{L^p(W)}^k$. Finally, by the sublinearity of $M$, 
\begin{eqnarray*} M(Rf)&\leq& \sum_{k=0}^\infty \frac{M^{k+1}f}{(2\|M\|_{L^p(W)})^k}\\
&\leq&2\|M\|_{L^p(W)}\sum_{k=1}^\infty \frac{M^kf}{(2\|M\|_{L^p(W)})^k}\\
&=&2\|M\|_{L^p(W)}Rf, 
\end{eqnarray*}
which proves (iii). 
\end{proof}

The final ingredient is the following estimate for $||M||_{L^p(W)}$. This is a probabilistic counterpart of the result of Buckley \cite{Bu} from the nineties, concerning the classical Hardy-Littlewood maximal operator. 

\begin{lemma}
For any $1<p<\infty$ there is a constant $c_p$ depending only on $p$ such that
$$ \|MX\|_{L^p(W)}\leq c_p[W]_{A_p}^{1/(p-1)}\|X_\infty\|_{L^p(W)}.$$
\end{lemma}
We will need an explicit formula for the constant $c_p$. It follows from  \cite{Os1} that 
$$ \|M\|_{L^p(W)}\leq \frac{p}{p-1-d(p,[W]_{A_p})},$$
where, for a given $1<p<\infty$ and $c\geq 1$, the constant $d(p,c)$ is the unique number in $[0,p-1)$ satisfying the equation
$$ c(1+d)(p-1-d)^{p-1}=(p-1)^{p-1}.$$
Consequently, we may write the more explicit bound
\begin{equation}\label{more-explicit}
\begin{split}
 \|M\|_{L^p(W)}&\leq \frac{p}{p-1-d(p,[W]_{A_p})}\\
&=\frac{p}{p-1}\big(d(p,[W]_{A_p})[W]_{A_p}\big)^{1/(p-1)}\\
&\leq \frac{p}{p-1}\cdot (p-1)^{1/(p-1)}[W]_{A_p}^{1/(p-1)}\leq \frac{pe}{p-1}[W]_{A_p}^{1/(p-1)}.
\end{split}
\end{equation}

We are ready for the main result of this section.

\begin{theorem}\label{thm-extrapolation}
Let $p_0\in (1,\infty)$ and $C,\,\kappa>0$  be  fixed parameters. Suppose that $f$, $g$ are nonnegative random variables such that for any $W\in A_{p_0}(mart)$ we have
$$ \|g\|_{L^{p_0}(W)}\leq C[W]_{A_{p_0}}^\kappa \|f\|_{L^{p_0}(W)}.$$
Then for all $1<p<\infty$ and all $W\in A_p(mart)$ we have
\begin{equation}\label{extrapolation_ineq}
 \|g\|_{L^p(W)}\leq L(W)\|f\|_{L^p(W)},
 \end{equation}
where
$$ L(W)=\begin{cases}
2^{1-p/p_0}C(2\|M\|_{L^p(W)})^{\kappa(p_0-p)}[W]_{A_p}^\kappa & \mbox{if }p<p_0,\\
2^{p'(1-p_0/p)/p_0}C(2\|M\|_{L^{p'}(W^{1-p'})})^{\kappa(p-p_0)/(p-1)}[W]_{A_p}^{\frac{\kappa(p_0-1)}{p-1}} & \mbox{if }p>p_0.
\end{cases}$$
In particular, we have $ L(W)\leq C_1[W]_{A_p}^{\kappa \max\{1,(p_0-1)/(p-1)\}}$ for some $C_1$ not depending on $W$.
\end{theorem}
\begin{proof}
\emph{Case $p<p_0$.} Let $Rf$ be the weight obtained from Rubio de Fraincia algorithm. Then, applying H\"older's inequality and Lemmas \ref{lem31} and \ref{lem32},
\begin{align*}
\E g^pW_\infty& =\E \big[g^pW_\infty (Rf)^{p(p-p_0)/p_0}(Rf)^{p(p_0-p)/p_0}\big]\\
&\leq \big(\E g^{p_0}W_\infty (Rf)^{p-p_0}\big)^{p/p_0}\big(\E (Rf)^pW_\infty\big)^{1-p/p_0}\\
&\leq 2^{p(1-p/p_0)}C^p[W(Rf)^{p-p_0}]_{A_{p_0}}^{\kappa p}\big(\E f^{p_0}W_\infty (Rf)^{p-p_0}\big)^{p/p_0}(\E f^pW_\infty\big)^{1-p/p_0}\\
&\leq (2^{1-p/p_0}C)^p[W]_{A_{p_0}}^{\kappa p}[Rf]_{A_1}^{\kappa p(p_0-p)} \E f^pW_\infty\\
&\leq (2^{1-p/p_0}C)^p(2\|M\|_{L^p(W)})^{\kappa p(p_0-p)}[W]_{A_{p_0}}^{\kappa p}\E f^pW_\infty,
\end{align*}
which is the desired estimate.

\emph{Case $p>p_0$.} By duality, we may write
$$ (\E g^pW_\infty)^{p_0/p}=\sup\left\{\left|\E g^{p_0}hW_\infty\right|\,:\,h\geq 0,\,\|h\|_{L^{p/(p-p_0)}(W)}\leq 1\right\}.$$
Fix a function $h$ as above. Then the function $H$, defined by the equality  $H^{p'}W_\infty^{1-p'}=h^{p/(p-p_0)}W_\infty$, satisfies $\|H\|_{L^{p'}(W_\infty^{1-p'})}\leq 1$. If we denote by $RH$ the $A_1$ weight obtained from Rubio de Francia algorithm, then, using the estimate $H\leq RH$, we get
\begin{align*}
\E g^{p_0}hW_\infty &\leq \E \big[g^{p_0}W_\infty^{(p_0-1)/(p-1)}(RH)^{(p-p_0)/(p-1)}\big]\\
&\leq C^{p_0}\left[W^{\frac{p_0-1}{p-1}}(RH)^{\frac{p-p_0}{p-1}}\right]_{A_{p_0}}^{\kappa p_0} \E f^{p_0}W_\infty^{(p_0-1)/(p-1)}(RH)^{(p-p_0)/(p-1)}\\
&\leq C^{p_0} \left([W]_{A_p}^{\frac{p_0-1}{p-1}}[RH]_{A_1}^{\frac{p-p_0}{p-1}}\right)^{\kappa p_0}\big(\E f^pW_\infty\big)^{p_0/p}\big(\E (RH)^{p'}W_\infty^{1-p'}\big)^{1-p_0/p}\\
&\leq 2^{p'(1-p_0/p)}C^{p_0} \left([W]_{A_p}^{\frac{p_0-1}{p-1}}(2\|M\|_{L^p(W)})^{\frac{p-p_0}{p-1}}\right)^{\kappa p_0}\big(\E f^pW_\infty\big)^{p_0/p}.
\end{align*}
This completes the proof.
\end{proof}

Equipped with the above theorem, we immediately obtain the proof of \eqref{mainin}.

\begin{proof}[Proof of \eqref{mainin}]
Fix a martingale $X$ and a weight $W$ as in the statement. We apply Theorem \eqref{thm-extrapolation} with $p_0=3$, $f=|X_\infty|$, $g=\langle X\rangle_\infty^{1/2}$ and $C=64$. If $p<3$, then, using \eqref{more-explicit}, the resulting constant $L(W)$ is bounded from above by
$$ 2^{1-p/3}\cdot 64 \cdot 2^{(3-p)/2}\left(\frac{pe}{p-1}\right)^{(3-p)/2}[W]_{A_p}^{1/(p-1)}\leq \frac{553}{p-1}[W]_{A_p}^{1/(p-1)}.$$
We used here an elementary fact (which is easy to prove) that the function
$$ p\mapsto 2^{1-p/3}\cdot 64 \cdot 2^{(3-p)/2}{(pe)^{(3-p)/2}}(p-1)^{(p-1)/2}$$
is decreasing on $[1,3]$ and smaller than $553$ for $p=1$. For $p>3$, we proceed similarly and bound $L(W)$ from above by
\begin{align*}
 2^{p'(1-3/p)/3}\cdot 64\cdot &2^{(p-3)/2(p-1)}\left(\frac{p'e}{p'-1}\right)^{(p-3)/2(p-1)}[W]_{A_p}^{1/2}\\
&=2^{5(p-3)/6(p-1)}\cdot 64\cdot (pe)^{(p-3)/2(p-1)}[W]_{A_p}^{1/2}.
\end{align*}
Since $(p-3)/(p-1)\leq 1$, this can be further bounded by
$$ 2^{5/6}\cdot 64\cdot e^{1/2} p^{1/2}[W]_{A_p}^{1/2}\leq 189 p^{1/2}[W]_{A_p}^{1/2}.$$
This is precisely the claim.
\end{proof}

\section{Proof of Theorem \ref{weights}}
We will now show how to deduce the weighted estimates for $g_*$ function from the martingale estimates studied in the preceding section. Recall that  $\D=\{x\in \bC: |z|<1\}$ is the unit disc in the complex plane and $\T=\partial D$ is its boundary, the unit circle. The associated Poisson kernel is given by the formula
$$P_z(e^{i\theta})=\frac{1-|z|^2}{|z-e^{i\theta}|^2}, \,\,\,\, z\in \D.$$
Suppose that $w$ is a weight, that is, a positive and integrable function on the unit circle. Denote by
$$
u_{w}(z)=\frac{1}{2\pi}\int_{\T} P_z(e^{i\theta}) w(e^{i\theta}) d\theta,\qquad z\in \D,
$$ 
the Poisson integral of $w$, i.e., the harmonic extension of $w$ to the unit disc $\D$. Recall that $w$ is said to be a Poissonian $A_p$ weight, if the condition \eqref{Poisson-T} is satisfied. Next, let $(B_t)_{t\geq 0}$ be Brownian motion in $\D$ started at the origin and let $\tau$ stand for its first exit time from $\D$.  Since $u_w$ is a harmonic function, the process $Y_t=u_{w}(B_{\tau\wedge t})$, $t\geq 0$, is a martingale terminating at the variable $Y_{\infty}=w(B_\tau)$. By the strong Markov property, we obtain that
$$\E_{0}\left(Y_{\infty}\,\,\big| \F_{\tau\wedge t}\right)=\E_{B_{\tau\wedge t}}\left(w(B_\tau)\right)=Y_t$$ 
and similarly for $Y_{\infty}^{-1/(p-1)}$:
$$
\E_0\left( Y_{\infty}^{-1/(p-1)}\big| \F_{\tau\wedge t}\right)=  \E_0\left(u_{w^{-1/p-1}}(B_{\tau})\,\big| \F_{\tau\wedge t}\right)=
\E_{B_{\tau\wedge t}}\left(u_{w^{-1/p-1}}(B_{\tau})\right).
$$
Comparing the conditions \eqref{probAp} and \eqref{Poisson-T}, we easily check  that $Y$ is a martingale $A_p$ weight if and only if $w$ is  a Poissonian $A_p$ weight. Actually, one  even has the equality
\begin{equation}\label{mart&T-equiv} 
[Y]_{A_p(\text{mart})}=[w]_{A_{p, \T}}. 
\end{equation}

We turn our attention to the Littlewood-Paley  $g_{*}$ function on the circle $\T$, given by \eqref{defgstar}. The crucial property of this square function, which gives the link to the probabilistic contents of the preceding sections, is that $g_*(f)$ can be represented as the conditional expectation of the square function  of the martingale $(u_f(B_{\tau\wedge t}))_{t\geq 0}$. More specifically, the application of It\^o's formula yields
$$
u_f(B_{\tau\wedge t})=u_f(0)+\int_0^{\tau\wedge t}\nabla u_f(B_s)\cdot dB_s
$$  
for all $t$, so the square function of this martingale equals
$$
\langle u_f(B)\rangle_{\tau\wedge t} =|u_f(0)|^2+\int_0^{\tau\wedge t}|\nabla u_f(B_s)|^2ds. 
$$
Now, it can be shown that
$$
g^2_*(f)(e^{i\theta})=\E_{0}^{\theta}\left(\int_0^{\tau}|\nabla u_f(B_s)|^2 ds\right)=\E_{0}\left( \int_0^{\tau}|\nabla u_f(B_s)|^2 ds \,\,\big| B_\tau=e^{i\theta} \right).
$$
We refer the reader to \cite[p.~650]{Ban1} for the detailed proof of this formula. Because the random variable $B_{\tau}$  is uniformly distributed on the circle $\T$ under $\p_0$, we may write, for any $p\geq 2$,
\begin{equation}\label{g-disc}
\begin{split}
\frac{1}{2\pi}\int_{\T} g^p_*(f)(e^{i\theta})w(e^{i\theta})d\theta&=\E_0\left(\left[\E_0\left( \int_0^{\tau}|\nabla u_f(B_s)|^2 ds\,\,\big| B_\tau\right)\right]^{p/2}w(B_\tau)\right)\\
&\leq \E_0\left( \left(\int_0^{\tau}|\nabla u_f(B_s)|^2 ds\right)^{p/2}\,w(B_\tau)\right). 
\end{split}
\end{equation}
for any weight $w$ on $\T$.

Equipped with the above representation, 
we apply it together with \eqref{mainin} to obtain 
\begin{align*}
\frac{1}{2\pi}\int_{\T} g^p_*(f)(e^{i\theta})w(e^{i\theta})d\theta&\leq \E_0\left( \left(\int_0^{\tau}|\nabla u_f(B_s)|^2 ds\right)^{p/2}\,w(B_\tau)\right)\\
&\leq \E_0 \bigg(\langle u_f(B)\rangle_{\tau}^{p/2}w(B_\tau)\bigg)\\
&\leq  K_p^p[w]_{A_{p,\T}}^{\max\{1/2,1/(p-1)\}p}\bigg(\E_0 |u_f(B_\tau)|^p w(B_\tau)\bigg)\\
&=K_p^p[w]_{A_{p,\T}}^{\max\{1/2,1/(p-1)\}p} \frac{1}{2\pi}\int_\T |f(e^{i\theta})|^pw(e^{i\theta})\mbox{d}\theta.
\end{align*} 
This is precisely the claim inequality in Theorem \ref{weights}.

The above result yields an immediate consequence for the Lusin area function and the Littlewood-Paley function associated with $f:\T\to \R$. Let us recall the necessary definitions. For $0<\alpha<1$, the Stoltz domain $\Gamma_{\alpha}(\theta)$ is the interior of the smallest convex set containing the disc $\{z\in \bC: |z|<\alpha\}$ and the point $e^{i\theta}$.  Then the Lusin area function (area integral) of $f$ is given by 
$$
A_{\alpha}(f)(e^{i\theta})=\left(\int_{\Gamma_{\alpha}(\theta)} |\nabla u_f(z)|^2 dz\right)^{1/2}. 
$$
Furthermore, the Littlewood-Paley function $g$ is an operator defined by  the formula
$$
g(f)(e^{i\theta})=\left(\int_0^1 (1-r)|\nabla u_f(re^{i\theta})|^2 dr\right)^{1/2}.
$$
It is not difficult to show that there are universal constant $C_{\alpha}$ and $C$ such that the pointwise inequalities $A_{\alpha}(f)(e^{i\theta})\leq C_{\alpha} g_*(f)(e^{i\theta})$ and $g(f)(e^{i\theta})\leq C g_*(f)(e^{i\theta})$ hold true.  Thus, \eqref{g*one} gives the following corollary with the same $K_p$.  

\begin{corollary} Let $p\geq 2$ and suppose that $w\in A_p(Poisson, \T)$ and  $f\in C(\T)$. Then we have
\begin{equation}
\|A_{\alpha}(f)(e^{i\theta})\|_{L^p_w(\mathbb{T})}\leq K_pC_{\alpha}([w]_{A_{p,\T}})^{\max\{1/2,1/(p-1)\}}\|f\|_{L^p_w(\mathbb{T})}
\end{equation}
and 
\begin{equation}\label{maininT4}
\|g(f)(e^{i\theta})\|_{L^p_w(\T)}\leq K_pC([w]_{A_{p,\T}})^{\max\{1/2,1/(p-1)\}}\|f\|_{L^p_w(\T)}.
\end{equation}
\end{corollary}

As already mentioned, the extensions to the various $g_{*}$, $A_{\alpha}$ and $g$ square functions on $\R^n$, $n\geq 1$,    are exactly as in \cite{BO}. The bounds are optimal with respect to $[w]_{A_p}$ and the constants $K_p$ has the best order as $p\to\infty$.  In the case of both $g_{*}$ and $g$, the bounds are independent of the dimension and for $A_{\alpha}$ the dependence  on $n$ and $\alpha$ is explicit as in Corollaries  5.2 and  5.3  in \cite{BO}.  The corresponding $L^p$ version of Theorem  6.1 in \cite{BO} for manifolds of non-negative Ricci curvature also holds.

\end{document}